\documentclass{amsart}
\usepackage{amssymb,latexsym,amsmath,amsthm,amsfonts}

\newtheorem{theorem}{Theorem}

\newtheorem{df}{Definition}
\newtheorem{prop}{Proposition}

 \title{Recurrence Relations and Determinants}

\begin{document}
\maketitle
\begin{center}Milan Janji\'c\end{center}
\begin{center}Department of Mathematics  and
Informatics,\end{center}
 \begin{center}University of Banja Luka, Republic of Srpska,
 BA\end{center}

\begin{abstract}
We examine relationships between two minors of order $n$  of some matrices of $n$ rows and $n+r$ columns. This is done through a
 class of determinants, here called $n$-determinants,
the investigation of which is our objective. We prove that
$1$-determinants are the upper Hessenberg determinants.
       In particular, we state several $1$-determinants each of which equals  a Fibonacci number.
We also  derive relationships among terms of sequences defined by the same recurrence equation independently of the initial conditions. A result generalizing the formula for the product of two determinants is obtained.
Finally, we prove that  the Schur functions may be expressed as $n$-determinants.

\end{abstract}


AMS Subject Classification 2008 11C20 11B37

\section{Introduction}

Let $A$ be a square matrix of order $n$ whose elements are in a
commutative ring $R.$ We define a block matrix
$A_r=[A|A_{n+1}|\cdots,|A_{n+r}]$ of $n$ rows and $n+r$
columns (where $A_{n+i},\;(i=1,2,\ldots,r)$ are vector columns) as follows:
\begin{equation}\label{r1} A_{n+j}=\sum_{i=1}^{n+j-1}p_{i,j}A_i,\;(j=1,2,\ldots,r),\;(p_{i,j}\in R).\end{equation}
For a  sequence  $1\leq j_1<j_2<\cdots<j_r<n+r$ of positive integers, we
 let $M=M(\widehat{j_1},\widehat{j_2},\ldots,\widehat{j_r})$ denote  the  determinant of the submatrix of $A_r,$ obtained  by deleting  columns  $j_1,j_2,\ldots,j_r$ of $A_r.$ The notation $\hat j$ means that the $j$th column is deleted.
Note that the last column of $A_r$ cannot be deleted.
The sign $\sigma(M)$ of $M$  is defined as
      \[\sigma(M)=(-1)^{nr+j_1+j_2+\cdots+j_r+\frac{(r-1)r}{2}}.\]
From  the coefficients  $p_{i,j},$ we form a matrix $P$ of $n+r-1$
 rows and $r$ columns as follows:
\begin{equation}\label{mpa}P=\left[\begin{array}{ccccc}
p_{1,1}&p_{1,2}&\cdots&p_{1,r-1}&p_{1,r}\\
p_{2,1}&p_{2,2}&\cdots&p_{2,r-1}&p_{2,r}\\
\vdots&\vdots&\cdots&\vdots&\vdots\\
p_{n,1}&p_{n,2}&\cdots&p_{n,r-1}&p_{n,r}\\
-1&p_{n+1,2}&\cdots&p_{n+1,r-1}&p_{n+1,r}\\
0&-1&\cdots&p_{n+2,r-1}&p_{n+2,r}\\
\vdots&\vdots&\cdots&\vdots&\vdots\\
0&0&\cdots&p_{n+r-2,r-1}&p_{n+r-2,r}\\
0&0&\cdots&-1&p_{r+n-1,r}
\end{array}\right].\end{equation}
The matrix $Q=Q(j_1,\ldots,j_r)$ is  obtained  by deleting $n-1$  rows of $P,$ the indices of which are different from $j_1,j_2,\ldots,j_r.$
Hence,
\begin{equation}\label{pp}Q(j_1,\ldots,j_r)=\left[\begin{array}{lllll}
p_{j_1,1}&\cdots&p_{j_1,t}&\cdots&p_{j_1,r}\\
p_{j_2,1}&\cdots&p_{j_2,t}&\cdots&p_{j_2,r}\\
\vdots&\cdots&\vdots&\cdots&\vdots\\
p_{j_r,1}&\cdots&p_{j_r,t}&\cdots&p_{j_r,r}
\end{array}\right].\end{equation}

\begin{df} We call  $\det Q$ an $n$-determinant.
\end{df}
  In this section, we prove that an $n$-determinant connects $\det A$ with an arbitrary minor of order $n$ of the matrix $A_r.$
In the remaining sections, we give several applications of this result.

 In Section 2, we consider $1$-determinants, which are the  upper Hessenberg determinants.
 Therefore, some important  mathematical objects may be represented as $1$-determinants.
 This is found to be the case for  the Catalan numbers, the  Bell numbers, the  Fibonacci numbers,
 the   Fibonacci polynomials, the generalized Fibonacci numbers, the Tchebychev polynomials of both kinds,
  the continuants, the derangements, the factorials   and the terms of any  homogenous linear recurrence equation.
 We also find several $1$-determinants, each of which equals a Fibonacci number.

The case $n=2$ is examined in Section 3. We show that
$2$-determinants produce some relationships  between two sequences
given by the same recurrence equation, with possibly different
initial conditions. In this sense, we prove a formula for the
Fibonacci polynomials from which  several well-known formulas
follow. For example, this is  the case with the Ocagne's formula
and the index reduction formula. Analogous formulas for the
Tchebychev polynomials are then stated. Also, we derive a result
for the continuants, generalizing the fundamental theorem of
convergents. Another result generalizes the standard recurrence
equation for the derangements.

In Section 4, we consider $3$-determinants and connect terms of three sequences given by the same
 recurrence equation. In particular, we obtain a result for the sequences satisfying the so-called  tribonacci recursion.

   We consider the case when $n$ is arbitrary in Section 5.  The form of an $n$-determinant in one such case is described. This leads to an extension of the formula for the product of two determinants.
  We finish the paper with a result of the type of  Jaccobi-Trudi's formula, expressing the Schur function as an $n$-determinant whose terms are  the elementary symmetric polynomials.

We begin with:
\begin{theorem}\label{th1} Let $1\leq j_1<\cdots<j_r<r+n$ be a sequence of positive integers.
Then,
\[M(\widehat{j_1},\ldots,\widehat{j_r})=\sigma(M)\cdot \det Q\cdot \det A.\]
\end{theorem}
\textbf{Proof.}
The proof is by induction on  $r.$
For  $r=1,$ we have  $1\leq j_1\leq r,$
since the case $j_1>r$ makes no sense. It follows that $M=M(\widehat{j_1},r+1).$
Taking $i=1$ in (\ref{r1}), we obtain
\[M(\widehat j_1,r+1)=\sum_{m=1}^n p_{m,1}M(\widehat j_1,m).\]
The sum on the right-hand side reduces to a single term when $m=j_1.$  We conclude that
\[M(\widehat{j_1},n+1)=p_{j_1,1}M(\widehat{j_1},j_1).\]
In $M(\widehat{j_1},j_1),$ we interchange the last column with the preceding one
and repeat this until the  $j_1$th column takes the $j_1$th place.
 For this, we need  $n-j_1$ interchanges. It follows  that
 \[M(\widehat j_1,j_1)=(-1)^{n+j_1}p_{j_1,1}\cdot \det A.\]
On the other hand, we obviously have
 $\sigma(M)=(-1)^{n+j_1},$ which proves the theorem for $r=1.$

Assume  that the theorem is true for  $1\leq k<r.$
The last column of the minor  $M(\widehat{j_1},\ldots,\widehat{j_r})$ is column $n+r$ of $A_r.$
The condition  (\ref{r1}) implies
\[ M(\widehat{j_1},\ldots,\widehat{j_r})=\sum_{m=1}^{n+r-1}p_{m,r}
 M(\widehat{j_1},\ldots,\widehat{j_r},m).\]

In the sum on the right-hand side the terms obtained for  $m\in\{j_1,\ldots,j_r\}$ remain.
They  are of the form:
\[S(t)=p_{j_t,r}M(\widehat{j_1},\ldots,\widehat{j_t},\ldots \widehat{j_r},j_t),\;(t=1,2,\ldots,r),\]
that is,
\[S(t)=(-1)^{n+r-1-j_t}p_{j_t,r}M(\widehat{j_1},\ldots,j_t,\ldots \widehat{j_r}).\]

Applying the induction hypothesis yields
 \[S(t)=(-1)^{n+t-1-j_t}\sigma(M(\widehat{j_1},\ldots,j_t,\ldots \widehat{j_r}))p_{j_t,r}
 Q(j_1,\ldots,\widehat{j_t},\ldots,j_r)\cdot \det A.\]
By a simple calculation, we obtain
\[(-1)^{n+t-1-j_t}\sigma(M(\widehat{j_1},\ldots,j_t,\ldots \widehat{j_r}))=(-1)^{r+t}
\sigma(M),\]
hence,
\[S(t)=\sigma(M)(-1)^{r+t}Q(j_1,\ldots,\widehat{j_t},\ldots,j_r)\cdot \det A.\]

Summing over all $t$ gives
\[\sum_{t=1}^rS(t)=\sigma(M)\bigg[\sum_{t=1}^r(-1)^{r+t}p_{j_t,r}Q(j_1,\ldots,\widehat{j_t},\ldots,j_r)\bigg]\cdot \det A.\]
The expression in the square brackets is the expansion  of (\ref{pp}) by elements of the last column, and the theorem is proved.

\section{$1$-determinants}
In this case,  $A$ is a matrix of order $1$, that is, an element of $R.$
We also have that $j_1=1,j_2=2,\ldots,j_r=r,$ hence,
the minor $M(\widehat{j_1},\widehat{j_2},\ldots,\widehat{j_r})$ must be of the form:
 $M(\hat 1,\hat 2,\ldots,\hat r).$ We easily obtain that $\sigma(M)=1.$
The matrix $P$ is as follows:
 \begin{equation}\label{mp1}P=\left[\begin{array}{rrrrrr}
p_{1,1}&p_{1,2}&p_{1,3}&\cdots&p_{1,r-1}&p_{1,r}\\
-1&p_{2,2}&p_{2,3}&\cdots&p_{2,r-1}&p_{2,r}\\
0&-1&p_{3,3}&\cdots&p_{3,r-1}&p_{3,r}\\
\vdots&\vdots&\vdots&\ddots&\vdots&\vdots\\
0&0&0&\cdots&p_{r-1,r-1}&p_{r-1,r}\\
0&0&0&\cdots&-1&p_{r,r}\end{array}\right].\end{equation}

 We see that  $Q=P.$ Therefore,  each $1$-determinant is an  upper Hessenberg
 determinant. Applying Theorem \ref{th1}, we obtain the following result.
\begin{prop}\label{tn1}
Let  $a_1,a_2,\ldots$ be a sequence in $R$ such that
\begin{equation}\label{fn}a_{1+r}=\sum_{i=1}^rp_{i,r}a_i.\end{equation}
Then,
\[a_{r+1}=a_1\det Q.\]
\end{prop}
This result is known. For instance, it follows from Theorem 4.20,\cite{hh}.

We give a number of examples for sequences given by the formula (\ref{fn}). Some of them are well-known.
In all examples, we take $a_1=1.$
\begin{enumerate}
\item[$1^\circ$] \textbf{ Catalan numbers.} We let $C_n$ denote the
$n$th Catalan number. If we take  $p_{i,j}=C_{j-i},$ then equation
(\ref{fn}) becomes
\[a_{1+r}=\sum_{i=1}^rC_{r-i}a_i.\]
The Segner's recurrence equation for Catalan numbers implies that $a_{r+1}=C_r.$
Hence, a way to write the Segner's formula in terms of determinants is
\[C_r=\left|\begin{array}{llllll}
C_0&C_1&C_2&\cdots&C_{r-2}&C_{r-1}\\
-1&C_0&C_1&\cdots&C_{r-3}&C_{r-1}\\
0&-1&C_0&\cdots&C_{r-4}&C_{r-3}\\
\vdots&\vdots&\vdots&\ddots&\vdots&\vdots\\
0&0&0&\cdots&C_0&C_1\\
0&0&0&\cdots&-1&C_0\end{array}\right|.\]

\item[$2^\circ$]\textbf{ Bell numbers.} If one takes $p_{i,j}={j-1\choose i-1}$ in (\ref{mp1}), then (\ref{fn}) becomes the recursion for the Bell numbers.
    Thus, a determinantal expression for the Bell number $B_r$ is
\[B_r=\left|\begin{array}{cccccc}
\left({0\atop 0}\right)&\left({1\atop 0}\right)&\left({2\atop 0}\right)&\cdots&\left({r-2\atop 0}\right)&\left({r-1\atop 0}\right)\\
-1&\left({1\atop 1}\right)&\left({2\atop 1}\right)&\cdots&\left({r-2\atop 1}\right)&\left({r-1\atop 1}\right)\\
0&-1&\left({2\atop 2}\right)&\cdots&\left({r-2\atop 2}\right)&\left({r-1\atop 2}\right)\\
\vdots&\vdots&\vdots&\ddots&\vdots&\vdots\\
0&0&0&\cdots&\left({r-2\atop r-2}\right)&\left({r-1\atop r-2}\right)\\
0&0&0&\cdots&-1&\left({r-1\atop r-1}\right)\end{array}\right|.\]

The order of the determinant equals $r.$

\item[$3^\circ$]\textbf{ Eigensequences for Stirling numbers.}  If
$\left\{{n\atop k}\right\}$ is the Stirling number of the second kind, and
$p_{i,j}=\left\{{j-1\atop i-1}\right\}$ in (\ref{mp1}),
 then (\ref{fn}) becomes the recursion for the so-called eigensequence $(E_1,E_2,\ldots)$ of the Stirling number of the second kind.
 Therefore,
\[E_r=\left|\begin{array}{cccccc}
\left\{{0\atop 0}\right\}&\left\{{1\atop 0}\right\}&\left\{{2\atop 0}\right\}&\cdots&\left\{{r-2\atop 0}\right\}&\left\{{r-1\atop 0}\right\}\\
-1&\left\{{1\atop 1}\right\}&\left\{{2\atop 1}\right\}&\cdots&\left\{{r-2\atop 1}\right\}&\left\{{r-1\atop 1}\right\}\\
0&-1&\left\{{2\atop 2}\right\}&\cdots&\left\{{r-2\atop 2}\right\}&\left\{{r-1\atop 2}\right\}\\
\vdots&\vdots&\vdots&\ddots&\vdots&\vdots\\
0&0&0&\cdots&\left\{{r-2\atop r-2}\right\}&\left\{{r-1\atop r-2}\right\}\\
0&0&0&\cdots&-1&\left\{{r-1\atop r-1}\right\}\end{array}\right|.\]

Note that analogous identity holds for the unsigned Stirling numbers of the first kind.
\item[$4^\circ$] \textbf{ Factorials}. Let $k$ be a positive integer. Consider the following $1$-determinant $D$  of order $r>1$:
\[D=\left|\begin{array}{ccccccc}
1&1&1&1&\cdots&1&1\\
-1&k&k&k&\cdots&k&k\\
0&-1&k+1&k+1&\cdots&k+1&k+1\\
0&0&-1&k+2&\cdots&k+2&k+2\\
\vdots&\vdots&\vdots&\ddots&\vdots&\vdots&\vdots\\
0&0&0&\cdots&-1&k+r-3&k+r-3\\
0&0&0&\cdots&0&-1&k+r-2\end{array}\right|.\]
In this case, the formula (\ref{fn}) becomes
\[a_{r+1}=1+\sum_{i=2}^r(k+i-2)a_i.\]
Subtracting the equation
\[a_{r+2}=1+\sum_{i=2}^{r+1}(k+i-2)a_i\] from the preceding, easily yields
\[a_{r+2}=(k+r)a_{r+1},\] which is the recursion for the falling factorials. Hence,
\[D=\frac{(k+r-1)!}{k!}.\]

\item[$5^\circ$] {\bf Derangements}.
We let $D_r$ denote the number of derangements of $r.$
The recurrence equation for the derangements is \[D_2=1,D_3=2,\;D_r=(r-1)(D_{r-2}+D_{r-1}),\;(r\geq 4).\]
We have  \[D_{r+1}=
\left|\begin{array}{rrrrrrr}
1&1&0&0&\cdots&0&0\\
-1&1&2&0&\cdots&0&0\\
0&-1&2&3&\cdots&0&0\\
0&0&-1&3&\cdots&0&0\\
\vdots&\vdots&\vdots&\ddots&\vdots&\vdots\\
0&0&0&\cdots&0&r-1&r\\
0&0&0&\cdots&0&-1&r\end{array}\right|.\]

\item[$6^\circ$]{\bf Fibonacci polynomials}.
In this case, we consider the recurrence equation  \[a_1=1,a_2=x,\;a_{k+1}=a_{k-1}+xa_k,\;(k\geq 2)\]
 for Fibonacci polynomials.
 Hence, for  Fibonacci polynomial $F_{r+1}(x)$  we have
\[F_{r+1}(x)=
\left|\begin{array}{rrrrrr}
x&1&0&\cdots&0&0\\
-1&x&1&\cdots&0&0\\
0&-1&x&\cdots&0&0\\
\vdots&\vdots&\vdots&\ddots&\vdots&\vdots\\
0&0&0&\cdots&x&1\\
0&0&0&\cdots&-1&x\end{array}\right|.\] The order of the
determinant  equals $r.$ Taking particulary $x=1,$ we obtain the well-known formula for Fibonacci numbers.

\item[$7^\circ$] {\bf Tchebychev polynomials of the first kind}.
The recurrence relation for the  Tchebychev polynomials of the first kind is
\[T_0(x)=1,\;T_1(x)=x,\; T_{k}(x)=-T_{k-2}(x)+2xT_{k-1}(x),\;(k>2).\]
Theorem \ref{tn1} now implies the following equation:
\[T_{r}(x)=\left|\begin{array}{rrrrrr}
x&-1&0&\cdots&0&0\\
-1&2x&-1&\cdots&0&0\\
0&-1&2x&\cdots&0&0\\
\vdots&\vdots&\vdots&\ddots&\vdots&\vdots\\
0&0&0&\cdots&2x&-1\\
0&0&0&\cdots&-1&2x\end{array}\right|.\]
The order of the determinant is $r.$
A similar formula holds for Tchebychev polynomials $U_r(x)$ of the second kind.
\item[$8^\circ$]{\bf Hermite polynomials.}
For the Hermite polynomials $H_r(x),$ we have the  following recurrence equation:
\[H_0(x)=1,\;H_1(x)=2x,\; H_{r+1}(x)=-2rH_{r-1}(x)+2xH_{r}(x),\;(r\geq 2).\]
Applying Theorem \ref{tn1},  we obtain the following expression:
\[H_{r}(x)=
\left|\begin{array}{cccccc}
2x&-2&0&\cdots&0&0\\
-1&2x&-4&\cdots&0&0\\
0&-1&2x&\cdots&0&0\\
\vdots&\vdots&\vdots&\ddots&\vdots&\vdots\\
0&0&0&\cdots&2x&-2(r-1)\\
0&0&0&\cdots&-1&2x\end{array}\right|.\]

\item[$9^\circ$] {\bf Continuants}.
Take in (\ref{fn}) $p_{k,k}=p_k,\;p_{k-1,k}=1,$ otherwise $p_{i,j}=0.$  We obtain the recursion:
\[a_1=1,p_1,\;a_{1+k}=a_{k-1}+p_ka_k,\;(k=2,\ldots).\]
The terms of this sequence are the continuants, and are denoted by $(p_1,p_2,\ldots,p_r).$
We thus obtain the following well-known formula:
\begin{equation}\label{kon}(p_1,p_2,\ldots,p_r)=\left|\begin{array}{llllll}
p_1&1&0&\cdots&0&0\\
-1&p_2&1&\cdots&0&0\\
\vdots&\vdots&\vdots&\vdots&\vdots&\vdots\\
0&0&0&\cdots&p_{r-1}&1\\
0&0&0&\cdots&-1&p_r\end{array}\right|.\end{equation}

\item[$10^\circ$] {\bf Linear homogenous recurrence equation}.
Let $b_1,b_2,\ldots,b_k$ be given elements of $R.$
Consider the sequence $1,a_2,a_3,\ldots$ defined as follows:
\[a_2=b_1,\ldots,a_{k+1}=b_k,\;a_{r+1}=\sum_{i=r-k+1}^rp_{i,r}a_i,\;(r>k).\]
We thus have a linear homogenous recurrence equation of order $k.$
From Theorem \ref{th1} follows
\[a_{r+1}=\left|\begin{array}{cccccccccc}
b_1&b_2&\cdots&b_k&0&0&\cdots&\cdots&0\\
-1&0&\cdots&0&p_{k+1,1}&0&\cdots&\cdots&0\\
0&-1&\cdots&0&p_{k+1,2}&p_{k+2,2}&\cdots&\cdots&0\\
\vdots&\vdots&\cdots&\vdots&\vdots&\vdots&\cdots&\cdots&\vdots\\
0&0&\cdots&0&p_{k+1,k-1}&p_{k+2,k-1}&\cdots&\cdots&0\\
0&0&\cdots&-1&p_{k+1,k}&p_{k+2,k}&\cdots&\cdots&0\\
\vdots&\vdots&\cdots&\vdots&\vdots&\vdots&\cdots&\cdots&0\\
0&0&\cdots&0&\vdots&\vdots&\cdots&\cdots&\cdots\\
0&0&\cdots&0&\vdots&\vdots&\cdots&-1&p_{k,r}
\end{array}\right|.\]

\item[$11^\circ$] {\bf Generalized Fibonacci numbers}.
Taking in the preceding formula that each $p_{i,j}$ equals $1,$ we obtain $k$-step Fibonacci numbers dependant on the initial conditions.
The standard $k$-step Fibonacci numbers $F^{(k)}_{n+k}$ are obtained for $b_1=b_2=\cdots=b_{k-1}=0,\;b_k=1.$
We thus have
\[F^{(k)}_{r+k}=\left|\begin{array}{rrrrrrrrr}
1&1&\cdots&1&0&\cdots&0&0&0\\
-1&1&\cdots&1&1&\cdots&0&0&0\\
0&-1&\cdots&1&1&\cdots&0&0&0\\
\vdots&\vdots&\cdots&\vdots&\vdots&\cdots&\vdots&\vdots&\vdots\\
0&0&\cdots&0&0&\cdots&\cdots&\vdots&\vdots\\
0&0&\cdots&0&0&\cdots&\cdots&-1&1
\end{array}\right|,\]
where the size of the determinant is $r+k.$

\item[$12^\circ$]{\bf Fibonacci numbers.} Consider the sequence
given by \[a_1=1,\;a_2=1,\;a_r=\sum_{i=1}^{r-2}a_i,\;(r>2).\] This, in fact, is a recursion for the Fibonacci numbers. To show this, we
first replace $r$ by $r+1$ to obtain
 \[a_{r+1}=\sum_{i=1}^{r-1}a_i.\]
 Subtracting two last equations yields
 \[a_{r+1}=a_{r}+a_{r+1},\] which is the standard recursion for the Fibonacci numbers.

Proposition \ref{tn1} implies
\[F_{r-1}=\left|\begin{array}{rrrrrrrr}
1&1&\cdots&1&\cdots&1&1\\
-1&0&\cdots&1&\cdots&1&1\\
0&-1&\cdots&1&\cdots&1&1\\
\vdots&\vdots&\cdots&\vdots&\cdots&\vdots&\vdots\\
0&0&\cdots&0&\cdots&0&1\\
0&0&\cdots&0&\cdots&-1&0
\end{array}\right|.\]
The order of the determinant equals $r.$

\item[$13^\circ$]{\bf Fibonacci numbers. } We define a matrix $Q_r=(q_{ij})$ of order $r$ as follows:
\[q_{ij}=\left\{\begin{array}{cc}
-1&i=j+1,\\i+j+1\,(\mbox{ mod }2)&i\leq j,\\
0&\mbox{ otherwise}.
\end{array}\right.\]
We find the recursion which, as in Proposition \ref{tn1}, produces this matrix.
Obviously, $a_1=1,\;a_2=1,\;a_3=1,\;a_4=2,$
and
\[a_{2r}=a_1+\cdots+a_{2r-1}.\] Also, \[a_{2r+2}=a_1+\cdots+a_{2r-1}+a_{2r+1}.\]
Subtracting two last equation yields $a_{2r+2}=a_{2r+1}+a_{2r}.$
Similarly, $a_{2r+1}=a_{2r}+a_{2r-1}.$  The recursion for the Fibonacci numbers is thus obtained. It follows that
$F_{r-1}=\det Q_r,\;(r>1).$

 \item[$14^\circ$]{\bf Fibonacci numbers with odd
indices}. Define a matrix $Q_r=(q_{ij})$ of order $r$ as follows:
\[q_{ij}=\left\{\begin{array}{cc}
-1&i=j+1,\\2&i=j,\\
1&i<j,\\ 0&\mbox{ otherwise}.
\end{array}\right.\]

In this case, we have the recursion $a_1=1,\;a_{2}=2,\;a_3=5,\;a_{r+1}=\sum_{i=1}^{r-1}+2a_r,\;(r\geq 2).$
From this, we easily obtain the recursion
\[a_{r+2}=3a_{r+1}-a_r.\] The identity 7, proved in \cite{bk}, shows that we have a recursion for the Fibonacci numbers with odd indices.
It follows that  $F_{2r+1}=\det Q_r.$

Note that we described in \cite{mil} a connection of this determinant with a particular kind of composition of natural numbers.

\item[$15^\circ$] {\bf Fibonacci numbers with even
indices.} For the matrix
\[Q_r=\left[\begin{array}{rrrrrr}
1&2&3&\cdots&r-1&r\\
-1&1&2&\cdots&r-2&r-1\\
0&-1&1&\cdots&r-3&r-2\\
\vdots&\vdots&\vdots&\vdots&\vdots&\vdots\\
0&0&0&\cdots&1&2\\
0&0&0&\cdots&-1&1\end{array}\right],\] the corresponding recursion has the form:
\[a_1=1,\;a_2=3,\;
a_{1+r}=\sum_{i=1}^r(r-i+1)a_i,\;(r\geq 2).\]
Also,
\[a_{2+r}=\sum_{i=1}^{r+1}(r-i+1)a_i+\sum_{i=1}^{r+1}a_i.\]
Subtracting this equation from the preceding, we obtain
\[a_{2+r}-a_{1+r}=\sum_{i=1}^{r+1}a_i.\]
In the same, way we obtain
\[a_{3+r}-a_{2+r}=\sum_{i=1}^{r+2}a_i.\]
Again, we subtract this equation from the preceding to obtain
\[a_{3+r}=3a_{2+r}-a_{1+r}.\]  This is a recursion for Fibonacci numbers, by   Identity 7 in  \cite{bk}. Taking into count the initial conditions, we have $\det Q_r=F_{2r}.$
\end{enumerate}

\section{$2$-determinants}
We consider the case $n=2.$ Assume additionally   that  $p_{i,j}=0,\;(j>i).$ Then, $P $ has at most three nonzero diagonals. Therefore, it may be written in the form:
\begin{equation}\label{pa2}P=\left[\begin{array}{llllll}
b_{1}&0&0&\cdots&0&0\\
c_{2}&b_{2}&0&\cdots&0&0\\
-1&c_{3}&b_{3}&\cdots&0&0\\
0&-1&c_{4}&\cdots&0&0\\
\vdots&\vdots&\vdots&\cdots&\vdots&\vdots\\
0&0&\vdots&\cdots&c_{r}&b_{r}\\
0&0&\vdots&\cdots&-1&c_{r+1}
\end{array}\right].\end{equation}
The corresponding $2$-determinant $\det Q$ is a lower triangular block determinant of the form:
\begin{equation}\label{mq}\left|\begin{array}{c|c}Q_{11}&0\\\hline Q_{12}&Q_{22}\end{array}\right|.\end{equation}
Here,   $Q_{11}$ is a lower triangular determinant lying in the first $k$ rows and the first $k$ columns of $Q.$
It follows that $Q_{11}=b_{1}\cdots b_{k}.$
 The order of $Q_{22}$ is $r-k$ and it is of the same form as the determinant of the  matrix $P$ in (\ref{mp1}).

As a consequence of Theorem \ref{th1}, we obtain
\begin{prop}\label{dvan} Let $(b_1,b_2,\ldots),\;(c_2,c_3,\ldots)$ be any two sequences. Let  \[(a_1^{(i)},a_2^{(i)},\ldots)\;(i=1,2)\] be two sequences defined by the same
recurrence equation of  the second order:
\[ a_r^{(i)}=b_{r-2}a^{(i)}_{r-2}+c_{r-1}a^{(i)}_{r-1},\;(r>2),\;(i=1,2).\] Then,
\[\left|\begin{array}{cc}a_{k+1}^{(1)}&a_{r+2}^{(1)}\\a_{k+1}^{(2)}&a_{r+2}^{(2)}\end{array}\right|=(-1)^{k}b_1\cdots b_k
\cdot d_{r-k+1}\cdot \left|\begin{array}{cc}a_{1}^{(1)}&a_{2}^{(1)}\\a_{1}^{(2)}&a_{2}^{(2)}\end{array}\right|,\]
 where
 \[d_1=1,\;d_2=c_{k+2},\;d_i=b_{k+i-1}d_{i-2}+c_{k+i}d_{i-1},\;(i>2).\]
\end{prop}

We illustrate the preceding proposition with  some examples.
\begin{enumerate}
\item[$1^\circ$] {\bf Fibonacci polynomials}. Take
$x_{1}^{(1)}=F_u(x),\;x_{2}^{(1)}=F_{u+1}(x),\;x_{1}^{(2)}=F_{v}(x),\;x_{2}^{(2)}=F_{v+1}(x),\;b_i=1,\;c_{i+1}=x,\;(i=1,2,\ldots).$

Note that, in this case, the $2$-determinant equals the Fibonacci polynomials $F_{r-k}(x).$
From Proposition \ref{dvan}, we obtain the following identity:
\begin{equation}\label{xxy}\left|\begin{array}{cc}F_{u+k}(x)&F_{u+r}(x)\\F_{v+k}(x)&F_{v+r}(x)\end{array}\right|=(-1)^{k}
F_{r-k}(x)\cdot
\left|\begin{array}{cc}F_{u}(x)&F_{u+1}(x)\\F_{v}(x)&F_{v+1}(x)\end{array}\right|.\end{equation}
Several well-known formulas may be obtained from this.

Taking $u=1,v=0$ yields
\[F_{k+1}(x)F_r(x)-F_k(x)F_{r+1}(x)=(-1)^{k}F_{r-k}(x),\]
which is the  Ocagne's identity for the Fibonacci polynomials.
Applying this identity on the right-hand side of (\ref{xxy}),  we obtain
\[\left|\begin{array}{cc}F_{u+m}(x)&F_{u+r}(x)\\F_{v+m}&F_{v+r}(x)\end{array}\right|=(-1)^{m+u+1}F_{r-m}(x)F_{v-u}(x).\]

We now may easily derive the index-reduction formula for the Fibonacci polynomials. Namely, replacing $m$ by $m-t$ and $r$ by $r-t,$ we get
\[\left|\begin{array}{cc}F_{u+k-t}(x)&F_{u+r-t}(x)\\F_{v+k-t}(x)&F_{v+r-t}(x)\end{array}\right|=(-1)^{k-t+u+1}F_{r-k}(x)
F_{v-u}(x).\]
Comparing the last two equations produces
\[\left|\begin{array}{cc}F_{u+k-t}(x)&F_{u+r-t}(x)\\F_{v+k-t}(x)&F_{v+r-t}(x)\end{array}\right|=(-1)^t
\left|\begin{array}{cc}F_{u+k}(x)&F_{u+r}(x)\\F_{v+k}(x)&F_{v+r}(x)\end{array}\right|,\]
which is the index reduction formula.

Note that such a formula for the Fibonacci numbers is proved  in \cite{Jo}.

\item[$2^\circ$] {\bf Fibonacci and Lucas polynomials}.
The Lucas polynomials $L_r(x)$ satisfy the same recurrence
relation as do the  Fibonacci polynomials with different initial
conditions. In this case, also, the $1$-determinant equals a Fibonacci polynomial.
We state two equations, one for mixed Lucas and  Fibonacci polynomials, another for Lucas polynomials:
\[\left|\begin{array}{cc}F_{u+k}(x)&F_{u+r}(x)\\L_{v+k}(x)&L_{v+r}(x)\end{array}\right|=(-1)^{k}F_{r-k}(x)\cdot
\left|\begin{array}{cc}F_{u}(x)&F_{u+1}(x)\\L_{v}(x)&L_{v+1}(x)\end{array}\right|,\]
and
\[\left|\begin{array}{cc}L_{u+k}(x)&L_{u+r}(x)\\L_{v+k}(x)&L_{v+r}(x)\end{array}\right|=(-1)^{k}F_{r-k}(x)\cdot
\left|\begin{array}{cc}L_{u}(x)&L_{u+1}(x)\\L_{v}(x)&L_{v+1}(x)\end{array}\right|.\]
\item[$3^\circ$] {\bf Tchebychev polynomials.}
Tchebychev polynomials of the first and second kind also satisfy the same recursion with different initial conditions.
Here, the $1$-determinant equals a Tchebychev polynomial of the second kind.
 We  state the following three identities, which are a consequence of Proposition \ref{dvan}"
  \[\left|\begin{array}{cc}U_{u+k}(x)&U_{u+r}(x)\\U_{v+k}(x)&U_{v+r}(x)\end{array}\right|=U_{r-k-1}(x)
\left|\begin{array}{cc}U_u(x)&U_{u+1}(x)\\U_v(x)&U_{v+1}(x)\end{array}\right|,\]

\[\left|\begin{array}{cc}T_{u+k}(x)&T_{u+r}(x)\\T_{v+k}(x)&T_{v+r}(x)\end{array}\right|=U_{r-k-1}(x)
\left|\begin{array}{cc}T_u(x)&T_{u+1}(x)\\T_v(x)&T_{v+1}(x)\end{array}\right|,\]

\[\left|\begin{array}{cc}U_{u+k}(x)&U_{u+r}(x)\\T_{v+k}(x)&T_{v+r}(x)\end{array}\right|=U_{r-k-1}(x)
\left|\begin{array}{cc}U_u(x)&U_{u+1}(x)\\T_v(x)&T_{v+1}(x)\end{array}\right|.\]

\item[$4^\circ$] {\bf Continued fractions}.
Up until now, the division was not used. We might therefore  assume that the elements of the concerned sequences belong to any commutative ring with $1.$ In this part, we suppose that they are positive real numbers.
Let $A_2$ be the identity matrix of order $2,$ and let $(c_1,c_2,\ldots)$ be an arbitrary sequence of positive real numbers.
   Form the matrix  $A_{r}$
by the following rule:
\[A_{2+k}=A_{k}+c_kA_{k+1},\;(k=1,2,\ldots,r).\]
It is easy to see that $A_r$  has the form:
     \[A_r=\left[\begin{array}{cccccccc}1&0&1&c_2&\ldots&(c_2,c_3,\ldots,c_{r})\\
0&1&c_1&(c_1,c_2)
&\ldots&(c_1,c_2,c_3,\ldots,c_{r})\end{array}\right],\] where $(c_m,c_{m+1},\ldots,c_p)$ are the continuants.
The $1$-determinant equals the  continuant $(c_{k+2},c_{k+3},\ldots,c_{r}).$

The fundamental recurrence relation for the continued fractions gives an expression  for the difference between two consecutive convergents.
Proposition \ref{dvan} allows us to derive a formula for the difference between two arbitrary convergents.
Thus, the following formula holds:
\[\left|\begin{array}{cc}(c_2,c_3,\ldots,c_{k})&(c_2,c_3,\ldots,c_{r})\\(c_1,c_2,c_3,\ldots,c_{k})&(c_1,c_2,c_3,\ldots,c_{r})
\end{array}\right|=(-1)^{k+1}(c_{k+2},c_{k+3},\ldots,c_{r}),\]
or, equivalently,
\[\frac{(c_1,c_2,c_3,\ldots,c_{r})}{(c_2,c_3,\ldots,c_{r})}-\frac{(c_1,c_2,c_3,\ldots,c_{k})}{(c_2,c_3,\ldots,c_{k})}=\]\[=
(-1)^{k+1}\frac{(c_{k+2},c_{k+3},\ldots,c_{r})}{(c_2,c_3,\ldots,c_{k})\cdot (c_2,c_3,\ldots,c_{r})},\]
with the convention that for $r=k+1$ the expression    $(c_{k+2},c_{k+3},\ldots,c_{r})$ equals $1.$

If $r<k+1,$ then the proof follows from Proposition \ref{dvan}.
If $r=k+1,$ then  we take $(c_{k+2},c_{k+3},\ldots,c_{m})=1,$ as then there is no matrix $Q_{22}.$
Hence, our formula becomes the continued fraction fundamental recurrence relation.

\item[$5^\circ$]{\bf Derangements.}
Take $A=\left[\begin{array}{cc}1&0\\1&1\end{array}\right].$ Let
the matrix $A_r$ be formed by the recursion
\[A_{2+r}=r(A_{1+r}+A_{r}),\;(r\geq 1).\]
 It is obvious that the $r$th element of the first row of $A_r$ equals $D_{r-1}.$ Also, the $r$th term of
 the second row of $A_r$ equals
  $(r-1)!.$
It follows that \[M(k+1,r+1)=\left|\begin{array}{cc}D_k&D_r\\k!&r!\end{array}\right|.\]
We thus obtain the following identity:
\[\left|\begin{array}{cc}D_k&D_r\\k!&r!\end{array}\right|=
(-1)^{k}k!
\left|\begin{array}{rrrrrr}
k+1&k+2&0&\cdots&0&0\\
-1&k+2&k+3&\cdots&0&0\\
0&-1&k+3&\cdots&0&0\\
0&0&-1&\cdots&0&0\\
\vdots&\vdots&\vdots&\vdots&\vdots\\
0&0&0&\cdots&r-1&r\\
0&0&0&\cdots&-1&r\end{array}\right|.\]
In particular, for $r=k+1,$ we have the standard recursion $D_{k+1}=kD_k+(-1)^k$ for the derangements.
\end{enumerate}
The preceding identities may be called the identities of order two, since they deal with determinants of order two.
Hence, we have identities of order two  of sequences satisfying a recurrence equation of order two.
We now proceed  to  derive  some identities of order two  for sequences satisfying a recurrence equation of order three.
Let $r$ be a positive integer, and let $(a_i),(b_{i+1}),(c_{i+2}),\;(i=1,2,\ldots)$ be any three sequences.
Let $A$ be an arbitrary matrix of order $2,$ and let the matrix $A_r$ be defined as follows:
 \[A_{3}=b_1A_1+c_2A_2,\;A_{3+j}=a_jA_j+b_{j+1}A_{j+1}+c_{j+2}A_{j+2},\;(1<j<r-2).\]
The corresponding matrix $P$ has the following form:
 \[P=\left[\begin{array}{rrrrrrr}
b_1&a_1&0&\cdots&0&0&0\\
c_2&b_2&a_2&\cdots&0&0&0\\
-1&c_3&b_3&\dots&0&0&0\\
0&-1&c_4&\dots&0&0&0\\
\vdots&\vdots&\vdots&\cdots&\vdots&\vdots&\vdots\\
0&0&0&\cdots&c_{r-1}&b_{r-1}&a_{r-1}\\
0&0&0&\cdots&-1&c_{r}&b_{r}\\
0&0&0&\cdots&0&-1&c_{r+1}
\end{array}\right].\]
The $2$-determinant  $\det Q$ is  obtained from $P$ by deleting  the $k+1$ row of $P,$
 where $(0\leq k<r).$  By $D(i_1,\ldots,i_t),$ we denote the minor of $P,$ the main diagonal of which is $(i_1,i_2,\ldots,i_t).$
 Hence, \[\det Q=D(b_1,\ldots,b_k,c_{k+2},\ldots,c_{r+1}).\]
Then, Theorem \ref{th1} gives an identity of order two
for terms of the matrix $A_r.$

We now investigate the structure of  $\det Q.$ Expanding by
elements of the first $k$ rows yields
\[\det Q=D(b_1,\ldots,b_k)D(c_{k+2},\ldots,c_{r+1})+a_kD(b_1,\ldots,b_{k-1})D(c_{k+3},\ldots,c_{r+1}).\]
On the other hand, we have
\[D(b_1)=b_1,\;
D(b_1,b_2)=\left|\begin{array}{rr}b_{1}&a_{1}\\c_{2}&b_{2}\end{array}\right|,\;D(b_1,b_2,b_3)=\left|
\begin{array}{rrr}b_{1}&a_{1}&0\\c_{2}&b_{2}&a_{2}\\
 -1&c_{3}&b_3\end{array}\right|.\]

Assume $k>3.$
 Expanding $D(b_1,\ldots,b_k)$ by the elements of the last column, we obtain
\[D(b_1,\ldots,b_k)=b_{k}D(b_1,\ldots,b_{k-1})-
a_{k-1}c_kD(1,\ldots,b_{k-2})-\]\begin{equation}\label{tri1}-a_{k-1}a_{k-2}D(b_1,\ldots,b_{k-3}).\end{equation}

Also, \[D(c_{k+2})=c_{k+2},\;
D(c_{k+2},c_{k+3})=\left|\begin{array}{rr}c_{k+2}&b_{k+2}\\-1&c_{k+3}\end{array}\right|,\] and
\[D(c_{k+2},c_{k+3},c_{k+4})=\left|\begin{array}{rrr}c_{k+2}&b_{k+2}&a_{k+2}\\-1&c_{k+3}&b_{k+3}\\
 0&-1&c_{k+4}\end{array}\right|.\]

If $k>3,$ then by expanding along the first column, we obtain
\[D(c_{k+2},\ldots,c_{r+1})=c_{k+3}D(c_{k+3},\ldots,c_{r+1})+\]
\begin{equation}\label{tri2}+b_{k+2}D(c_{k+4},\ldots,c_{r+1})+a_{k+2}D(c_{k+5},\ldots,c_{r+1}).\end{equation}

We have thus proved
\begin{prop} The $2$ determinant $\det Q$ is uniquely determined by the recurrence equations  (\ref{tri1}) and (\ref{tri2}).
\end{prop}

We illustrate the preceding considerations by the so-called \textbf{tribonacci} numbers. We assume that all $a$'s, $b$'s, and $c$'s equal $1.$
Then,
\[D(b_1)=1,\;D(b_1,b_2)=0,\;d(b_1,b_2,b_3)=-2,\]
and, for $s>3,$
\[D(b_1,b_2,\ldots,b_s)=D(b_1,\ldots,b_{s-1})-D(b_1,\ldots,b_{s-2})
-D(b_1,\ldots,b_{s-3}).\] This recursion designates the so-called \textbf{reflected tribonacci numbers}\\ (A057597, \cite{slo}).
We denote these numbers by $RT_i(1,0,-2),\;(i=1,2,\ldots).$
Also,
\[D(c_1)=1,\;D(c_1,c_2)=2,\;d(c_1,c_2,c_3)=4,\]
and, for $s>3,$
\[D(c_1,c_2,\ldots,c_s)=D(c_1,\ldots,c_{s-1})+D(c_1,\ldots,c_{s-2})
+D(c_1,\ldots,c_{s-3}).\] This is a recursion for tribonacci numbers, denoted by $T_i(1,2,4),\;(i=1,2,\ldots).$

Hence, our $2$-determinant consists of tribonacci and reflected tribonacci numbers. On the other hand, if $A$ is the identity matrix of order $2,$ then the first row of $A_r$
consists of tribonacci numbers $T_i(1,0,1),\;(i=1,2,\ldots),$ and the second row consists of the standard tribonacci numbers \\ $T_i(0,1,1),\;(i=1,2,\ldots).$ As a consequence of Theorem \ref{th1}, we have the following identity:
\[\left|\begin{array}{rr}T_{k+1}(1,0,1)&T_{r+2}(1,0,1)\\T_{k+1}(0,1,1)&T_{r+2}(0,1,1)\end{array}\right|=(-1)^k
\left|\begin{array}{rr}RT_{k}(1,0,-2)&-RT_{k-1}(1,0,-2)\\T_{r-k}(1,2,4)&T_{r-k-1}(1,2,4)\end{array}\right|.\]
Note that last rows in the preceding determinants consist of the standard tribonacci numbers.

   \section{$3$-determinants}
In this section, the order of the matrix $A$ being $3,$ we
deal with determinants of order $3.$ It may be said that  the
corresponding identities are all of order $3$.    We investigate in
detail the particular case when $p_{i,j}=0,\;(j>i).$  Then, the
matrix $P$ may have at most four nonzero diagonals.
 We have \[P=\left[\begin{array}{rrrrrrr}
a_1&0&0&\cdots&0&0&0\\
b_2&a_2&0&\cdots&0&0&0\\
c_3&b_3&a_3&\dots&0&0&0\\
-1&c_4&b_4&\dots&0&0&0\\
0&-1&c_5&\cdots&0&0&0\\
\vdots&\vdots&\vdots&\cdots&\vdots&\vdots&\vdots\\
0&0&0&\cdots&c_r&b_r&a_r\\
0&0&0&\cdots&-1&c_{r+1}&b_{r+1}\\
0&0&0&\cdots&0&-1&c_{r+2}
\end{array}\right].\]
The corresponding $3$-determinant $\det Q$   is  obtained by deleting
rows $k+1$ and $m+1$ of $P,$  where $(0\leq k<m\leq  r+1).$ The matrix $Q$ is a lower
triangular block matrix of the form (\ref{mq}),
  with $\det Q_{11}=a_1\cdot a_2\cdots a_k.$
 The order of the matrix  $Q_{22}$ is $r-k.$

 We denote $\det Q_{22}=D_k(s,r),$ where $s=m-k.$ Note that $s\geq 1.$
 For the columns of $A_r$ we have the following recursion:
  \begin{equation}\label{rekk}A_{3+i}=a_iA_{i}+b_{i+1}A_{1+i}+c_{i+2}A_{i+2},\;(i\geq 1).\end{equation}
  It follows that the sequences in the rows of $A_r$ satisfy a recurrence equation of order $3,$
 with the initial conditions given by the rows of $A.$
The set $\{j_1,\ldots,j_r\}$ equals
$\{1,2,\ldots,k,k+2,\ldots,m,m+2,\ldots,r+2\}.$  A simple
calculation shows that $\sigma(M)=(-1)^{m+k+1}.$ As a consequence of  Theorem \ref{th1}, we have
\begin{prop}\label{cetiri} Let $A$ be any matrix of order $3$, let $(a_i),(b_{i+1}),(c_{i+2}),\;(i=1,2,\ldots)$
 be any sequences.
Then,
\[M(k+1,m+1,r+2)=(-1)^{m+k+1}a_1\cdots a_k\det Q \cdot\det A.
\]
\end{prop}

We denote $\det Q=D_k(s,r).$ Note that $s\geq 1.$  Now investigate the structure of the array $D_k(s,r).$
The matrix $Q_{22}$ has at most  five nonzero   diagonals.
Assume that $s\geq 3.$

1) The main  diagonal of $Q_{22}$ is
 \[b_{k+2},\ldots,b_{k+s},\widehat{k+s+1},c_{k+s+2},\ldots,c_{r+2},\]
where there are $s-1$ of $b$'s and $r+1-k-s$ of $a$'s\\
2) The first superdiagonal  is
\[a_{k+2},\ldots,a_{k+s},\widehat{k+s+1},b_{k+s+2}\ldots,b_{r+1},\] where
there are  $s-1$ of $a$'s and $r-k-s$ of $b$'s.\\
 3) The first subdiagonal is
 \[c_{k+3},\ldots,c_{k+s},\widehat{k+s+1},-1,\ldots,-1,\]
where  there are  $s-2$ of $c$'s and $r+1-k-s$ of $-1$'s.\\
 4) The second superdiagonal is
  \[0,\ldots,0,\widehat{k+s+1},a_{k+s+2},\ldots,a_r,\]
  where  there are $s-1$ of $0$'s and $r-k-s-1$ of $a$'s.\\
 5) The second subdiagonal is
 \[-1,\ldots,-1,\widehat{k+s+1},0,\dots,0,\]
 where  there are $s-3$ of $-1$'s and $r+1-k-s$ of $0$'s.

We prove that the array  $D_k(s,r)$ may be determined recursively.

{\bf Case $s=1$}. We begin with
\[D_k(1,k)=1,\;D_k(1,k+1)=c_{k+3},\;D_k(1,k+2)= \left|\begin{array}{rr}c_{k+3}&b_{k+3}
 \\-1&c_{k+4}\end{array}\right|.\]
 For $t>2,$ expanding $D_k(1,k+t)$ by elements from the last row yields the following recursion:
\[D_k(1,k+t)=c_{k+t+2}D_k(1,k+t-1)+b_{k+t+1}D_k(1,k+t-2)+\]\begin{equation}\label{rek1}+a_{k+t}D_k(1,k+t-3).\end{equation}

{\bf Case $s=2$}. We have
\[D_k(2,k+1)=b_{k+2},\;D_k(2,k+2)= \left|\begin{array}{rr}b_{k+2}&a_{k+2}
 \\-1&c_{k+4}\end{array}\right|,\]\[
 D_k(2,k+3)=\left|\begin{array}{rrr}b_{k+2}&a_{k+2}&0\\-1&c_{k+4}&b_{k+4}\\
 0&-1&c_{k+5}\end{array}\right|.\]
For $t\geq 4,$  we calculate $D_k(2,k+t)$ by the recursion (\ref{rek1}).

{\bf Case $s=3$}. We now have $D_k(3,k+2)=b_{k+2},$
\[D_k(3,k+3)= \left|\begin{array}{rr}b_{k+2}&a_{k+2}
 \\c_{k+3}&b_{k+3}\end{array}\right|,\;
 D_k(3,k+4)=
 \left|\begin{array}{rrr}b_{k+2}&a_{k+2}&0\\c_{k+3}&b_{k+3}&a_{k+3}\\
 0&-1&c_{k+5}\end{array}\right|,
 \]
 \[D_k(3,k+5)=\left|\begin{array}{rrrr}b_{k+2}&a_{k+2}&0&0\\c_{k+3}&b_{k+3}&a_{k+3}&0\\
 0&-1&c_{k+5}&b_{k+5}\\0&0&-1&c_{k+6}\end{array}\right|.\]
For $t>5,$ we calculate $D_k(3,k+t)$ again by the recursion (\ref{rek1}).

{\bf Case $s\geq 4.$} The minors  $D_k(s,k+s-1),\ldots,D_k(s, k+2s-1)$ may be obtained as follows:
\[D_k(s,k+s-1)=b_{k+2},\;
  D_k(s,k+s)=\left|\begin{array}{rr}b_{k+2}&a_{k+2}
 \\c_{k+3}&b_{k+3}\end{array}\right|,\]\[
 D_k(s,k+s+1)=\left|\begin{array}{rrr}b_{k+2}&a_{k+2}&0\\c_{k+3}&b_{k+3}&a_{k+3}\\
 -1&c_{k+4}&b_{k+4}\end{array}\right|
 .\]  When $1<t\leq s-1,$  we have the following recursion:
\[D_k(s,k+s+t)=b_{t+k+2}D_k(s,k+s+t-1)-a_{t+k+1}c_{t+k+2}D_k(s,k+s+t-2)-\]
 \begin{equation}\label{post}-a_{t+k+1}a_{t+k}D_k(s,k+s+t-3).\end{equation}

 Next, we have
  \[D_k(s,k+2s)=c_{s+k+2}D_k(s,k+2s-1)+a_{s+k}D_k(s,k+2s-2).\]
  If $s<r-k,$ then
 \[ D_k(s,k+2s+1)=c_{s+k+3}D_k(s,k+2s)+b_{s+k+2}D_k(s,k+2s-1).\]
 If $s+1<r-k,$ then for $t$, where $s+1<t\leq r-k,$ we have the recursion (\ref{rek1}).

The recursion with respect to $k$ is backward. The minimal value that $r$ can take is $r=k.$ Then,
\[D_k(1,k)=1,\;D_k(2,k+1)=D_{k}(3,k+2)=b_{k+2}.\]
 Assume that $s>3.$ Expanding $D_k(s,r)$ by elements of the first row, we obtain the following recursion:
\[D_k(s,r)=b_{k+2}D_{k+1}(s-1,r-1)-a_{k+2}c_{k+3}D_{k+2}(s-2,r-2)-\]
\begin{equation}\label{pok}-a_{k+2}a_{k+3}D_{k+3}(s-3,r-3).
   \end{equation}
We have thus proved
\begin{prop} The array $D_k(s,r)$ is uniquely determined by the formulas (\ref{rek1}), (\ref{post}), and (\ref{pok}).
\end{prop}
 We state some examples.
\begin{enumerate}
\item[$1^\circ$] All $a$'s equal $0.$
 It follows from(\ref{cetiri})  that all minors $M(k+1,m+1,r+3)$ are zeros, except the case $k=0,$
when we have the same situation as in the preceding section.

\item[$2^\circ$]  All $b$'s equal $0,$ and all $a$'s and $c$'s equal $1$.
The formula (\ref{rekk}) has the form:
\[A_{3+i}=A_{i}+A_{i+2},\;(i\geq 1).\]
If $A$ is the identity matrix, then the rows of $A_r$ make the so-called {\bf middle sequence}((A000930,\cite{slo})).
For a fixed $k$ the first three rows of the array $D_k(s,r)$ are obtained by the recursion (\ref{rek1}), hence they are also formed by the numbers of the middle sequence. If $s>3,$ then the first $s-1$ elements in row $s$  are obtained by the recursion
 (\ref{post}). The remaining terms are again obtained from (\ref{rek1}).
Therefore, (\ref{cetiri}) becomes an identity for the numbers of the middle sequence.

\item[$3^\circ$]  All $c$'s equal $0$, and all $a$'s
 and $b$'s equal $1$.
In this case, we have
\[A_{3+i}=A_{i+2}+A_{i+3},\;(i\geq 1),\] which is
the recursion for the sequence of {\bf Padovan} numbers.
From Proposition \ref{cetiri}, we obtain an assertion for the Padovan numbers.

\item[$4^\circ$]  All $a$'s, $b$'s and $c$'s equal $1$.
The rows of $A_r$ are made of tribonaci numbers, with the initial conditions given by the rows of $A.$
The area $D_k(s,r)$ is also made of the tribonacci numbers, with the  initial conditions given by the first
 three values $s_1,s_2,s_3$ in row $s$ of the array $D_k(s,r).$
Assume that $A$ is the identity matrix of order $3.$ We then have
\begin{prop}
Let $0\leq k<m<n+2$ be integers. Then,
\[
\left|\begin{array}{rrr} T_k(1,0,0)&T_m(1,0,0)&T_{r+2}(1,0,0)\\
T_k(0,1,0)&T_m(0,1,0)&T_{r+2}(0,1,0)\\
T_k(0,0,1)&T_m(0,0,1)&T_{r+2}(0,0,1)
\end{array}\right|=(-1)^{m+k+1}T_r(s_1,s_2,s_3).\] \end{prop}
\end{enumerate}

\section{$n$-determinants}
We first consider $n$-determinants arising from a matrix $P$ of the following form:
\begin{equation}\label{mpp}P=\left[\begin{array}{ccccc}
p_{1,1}&p_{1,2}&\cdots&p_{1,r-1}&p_{1,r}\\
p_{2,1}&p_{2,2}&\cdots&p_{2,r-1}&p_{2,r}\\
\vdots&\vdots&\cdots&\vdots&\vdots\\
p_{n,1}&p_{n,2}&\cdots&p_{n,r-1}&p_{n,r}\\
-1&0&\cdots&0&0\\
\vdots&\vdots&\cdots&\vdots&\vdots\\
0&0&\cdots&-1&0
\end{array}\right].\end{equation}
We conclude that vector columns of $A_r$ are obtained as linear combinations of the columns of
$A$ whose coefficients are elements of the columns of $B,$ where $B$ is the submatrix of $P$ lying in its first $n$ rows.
In other words, we have  \[A_r=[A|AB].\]

We write $P$ in the form of a block matrix $P=\left[\frac{B}{C}\right],$ where $B$ consists of the first $n$ rows  of $P.$
 Hence, each  row of $C$ has one term equal to $-1,$ and all other terms equal to $0.$
Let $D$ be a $n$-determinant which is obtained  by deleting rows
$j_1,j_2,\ldots,j_r$ of $P.$  Let $k$ be a nonnegative integer such that
\[\{j_1,j_2,\cdots,j_k\}\subseteq\{1,2,\ldots,n\},\]\[
\{j_{k+1},j_{k+2},\cdots,j_{r}\}\subseteq\{n+1,n+2,\ldots,n+r-1\}.\]
We first consider the case when one of
$\{j_1,j_2,\cdots,j_k\},\;\{j_{k+1},j_{k+2},\cdots,j_{r}\}$  is empty.
When $\{j_{k+1},j_{k+2},\cdots,j_{r}\}=\emptyset,$ we have
\begin{prop} The $n$-determinant  $D$ is the minor lying in rows $j_1,\ldots,j_r$ of $B.$ If $r=n,$
then $D=det(B).$
\end{prop}
 \begin{proof} In this case, we have $k=r,$ hence $r\leq n.$ It follows that $D$ is a minor of $B$ lying in rows
 $j_1,\ldots,j_r$ of $B.$ If $r=n,$ then $D=\det B,$ and (\ref{minor}) becomes \[\det(AB)=\det A\cdot \det B.\]
\end{proof}
In the sense of the this equation, the above considerations may be regarded as an extension of the formula for the product of two determinants.

Assume that  $\{j_{1},\ldots,j_{k}\}$ is empty. This means that $M$ consists of the first $n-1$ rows of the matrix $A$ and the last columns of the matrix $AB.$ On the other hand, $D=b_{n,r}$ and  equation (\ref{minor}) becomes trivial.

Therefore, we may assume that both sets are not empty.
\begin{prop}\label{ppa} Let $A,B,A_r,\;j_1,\ldots,j_r,k$ be as above.
 Denote by $D$ the minor of $B$ lying in rows
$j_1,j_2,\ldots,j_k,$ and in the columns the indices of which are
different from $j_{k+1}-n,\ldots,j_{r}-n.$ If $M$ is the
determinant of the matrix which is obtained by deleting columns
$j_1,j_2,\ldots,j_r$ of $[A|AB],$ then
\begin{equation}\label{minor}M= (-1)^{kn+j_1+\cdots+j_k+\frac{(r-k)(r+k+3)}{2}}\cdot D\cdot\det A.\end{equation}
\end{prop}
\begin{proof}
  We may write $D$ in
the form $D=\det \left[\frac{Q_1}{Q_2}\right],$ where $Q_1$
consists of rows $j_1,j_2,\ldots,j_k$ of $B,$ and
  $Q_2$ consists of rows of $\{j_{k+1},\ldots,j_{r}\}$ of $P.$  The matrix $Q_2$ has $r-k$ rows, which have  only one term equal to $-1,$ and all other terms equal to  $0.$
  It  follows  that $Q_2$ has only one submatrix $X=(x_{ij})$ of order $r-k,$ the determinant of which  is not $0.$
  It is clear that $X$ is a diagonal matrix with $-1$'s on the main diagonal. Therefore,
  its determinant equals $(-1)^{r-k}.$ We conclude that the expansion of $D$
  by elements of the last $r-k$ rows has only one term.
We calculate the sum of the indices of rows and columns of $D,$ in which lies the minor $X$.
 The sum of the indices of rows is \[(k+1)+(k+2)+\cdots+r=k(r-k)+\frac{(r-k)(r-k+1)}{2}.\]
 In the matrix $P,$ the $(-1)$'s lie in rows $n+i$ and columns $i.$ The indices of the columns in $P$ and $D$ are the same, hence
 the sum of the  indices of the columns containing $X$ is
 \[(j_{k+1}-n)+(j_{k+2}-n)+\cdots+(j_r-n)=j_{k+1}+\cdots+j_r-(r-k)n.\]
Therefore,
\[D=(-1)^{j_{k+1}+\cdots+j_r-(r-k)n+\frac{(r-k)(r+k+3)}{2}}Y,\]
where $Y$ is the complement minor of $X$ in $D.$
 We saw that $Y$ lies in rows
$j_1,j_2,\ldots,j_k$ of $P.$ We now see that it lies in the
columns of $P,$ the indices of which are different from
$\{j_{k+1}-n,\ldots,j_{r}-n\}.$ Finally, we have
\[\sigma(M)\cdot (-1)^{j_{k+1}+\cdots+j_r-(r-k)n+\frac{(r-k)(r+k+3)}{2}}=(-1)^{kn+j_1+\cdots+j_k+\frac{(r-k)(r+k+3)}{2}},\] and the proposition is proved.
\end{proof}

We finish the paper with a formula in which
the generalized Vandermonde determinant is expressed  in terms of
the elementary symmetric polynomials.
 One such formula  is the Jacobi-Trudi's formula, Theorem 7.16.1,\cite{sta}.

 The determinant of the form  \[M(k_1,\ldots,k_n)=\left|
        \begin{array}{ccccc}
          x_1^{k_1} & x_1^{k_2} & x_1^{k_3} &\cdots&x_1^{k_n} \\
          x_2^{k_1} & x_2^{k_2} & x_2^{k_3} &\cdots&x_2^{k_n} \\
          \vdots &\vdots & \vdots &\vdots \\
           x_n^{k_1} & x_n^{k_2} & x_n^{k_3} &\cdots&x_n^{k_n} \\
        \end{array}
      \right|,\;(0\leq k_1<\cdots<k_n)\] is called the generalized Vandermonde determinant.
      The expression \[\frac{M(k_1,\ldots,k_n)}{ M(1,\ldots,n)}\]
 is called the Schur function.

Define a polynomial $f_{n}(x)$ such that
\[f_n(x)=(x-x_1)(x-x_2)\cdots (x-x_n).\]
Expanding the right side, we have
\[f_{n}(x)=x^n-p_{n}x^{n-1}-\cdots-p_1,\]
where
\begin{equation}\label{}p_{n-k}=(-1)^{k}\sigma_{k+1}(x_1,x_2,\ldots,x_n),\;(k=0,1,2,\ldots,n-1).\end{equation}
Here $\sigma_{k+1}(x_1,x_2,\ldots,x_n)$ is the elementary symmetric polynomial of order $k+1$ of $x_1,\ldots,x_n.$
It follows that
\begin{equation}\label{usp}x_i^n=p_1+p_2x_i+\cdots+p_nx_i^{n-1},\;(i=1,2,\ldots,n).\end{equation}
Consider the following matrix:
\[V=\left[
        \begin{array}{ccccc|ccc}
          1 & x_1 & x_1^2 &\cdots&x_1^{n-1}& x_1^n&\cdots&x_1^{k_n}\\
         1 & x_2 & x_2^2 &\cdots&x_2^{n-1}& x_2^n&\cdots&x_2^{k_n}\\
          \vdots &\vdots & \vdots &\vdots &\vdots&\vdots\\
          1 & x_n & x_n^2 &\cdots&x_n^{n-1}& x_n^n&\cdots&x_n^{k_n}
        \end{array}
      \right].\]

According to (\ref{usp}), we have
\begin{equation}\label{jjj3}
V_{n+k}=\sum_{j=1}^{n}p_jV_{j+k-1},\;(k=1,2,\ldots,k_n-n).\end{equation}

In  view of (\ref{jjj3}), the corresponding  matrix  $P$ in (\ref{mpa}) is a $k_n$ by $k_n+1-n$ matrix whose elements are
the elementary symmetric polynomials of $x_1,x_2,\ldots,x_n.$

For $\sigma(M),$ one easily obtains that
\[\sigma(M)=(-1)^{\frac{n(n-1)}{2}+\sum_{i=1}^{n-1}k_i}.\]
Also, the corresponding $n$-determinant $\det Q$  is obtained by deleting the rows $k_1,k_2,\ldots,k_{n-1}$ of $P.$

By Theorem \ref{th1}, we obtain
\begin{prop} The following formula is true
\[ M(k_1,k_2,\ldots,k_n)=\sigma(M)\cdot \det Q\cdot M(1,2,\ldots,n).\]
\end{prop}

Note that the expression $\sigma(M)\cdot \det Q$ equals the Schur function.

\label{}





\bibliographystyle{model1a-num-names}

\end{document}